\documentclass[12pt]{amsart}
\theoremstyle{plain}

\newtheorem{theorem}{Theorem}[section]

\newtheorem{lemma}[theorem]{Lemma}
\newtheorem{proposition}[theorem]{Proposition}
\newtheorem{corollary}[theorem]{Corollary}

\theoremstyle{definition}

\newtheorem{definition}[theorem]{Definition}
\newtheorem{remark}[theorem]{Remark}
\newtheorem{example}[theorem]{Example}

\newcommand{\ep}{\varepsilon}

\newcommand{\vf}{\varphi}

\newcommand{\cv}{\mathrm{cv}}

\newcommand{\length}{\mathrm{length}}

\newcommand{\dist}{\mathrm{dist}}

\newcommand{\R}{\mathbb{R}}
\newcommand{\N}{\mathbb{N}}

\newcommand{\Ha}{\mathcal{H}}

\renewcommand{\epsilon}{\varepsilon}
\renewcommand{\phi}{\varphi}

\newcommand{\cal}{\mathcal}

\newcommand{\reach}{{\rm reach}\,}

\newcommand{\Crit}{{\rm Crit}}

\usepackage{amssymb}
\usepackage{amsmath}
\usepackage{amsthm}

\def \cal{\mathcal}

\begin{document}

\title{Critical values and level sets of distance functions in Riemannian, Alexandrov and Minkowski spaces}

\author{Jan Rataj}

\author{Lud\v ek Zaj\'\i\v cek}
\address{Charles University\\
Faculty of Mathematics and Physics\\
Sokolovsk\'a 83\\
186 75 Praha 8\\
Czech Republic}

\email{rataj@karlin.mff.cuni.cz}
\email{zajicek@karlin.mff.cuni.cz}



 \thanks{The research was supported by the grant MSM 0021620839 from
 the Czech Ministry of Education.
The second author was also supported by the grant
GA\v CR 201/09/0067.  
}

\begin{abstract}

Let $F \subset \R^n$ be a closed set and $n=2$ or $n=3$. S. Ferry (1975) proved that then, for almost all $r>0$,
 the level set  (distance sphere, $r$-boundary) $S_r(F):= \{ x \in \R^n:\ \dist(x,F) = r\}$ is a topological $(n-1)$-dimensional
  manifold. This result was improved by J.H.G. Fu (1985). We show that Ferry's result is an easy consequence
   of the only fact that the distance function  $d(x)=  \dist(x,F)$ is locally DC and has no stationary point
   in $\R^n \setminus F$. Using this observation, we show that Ferry's  (and even Fu's) result extends to sufficiently smooth normed linear spaces $X$ with $\dim X \in \{2,3\}$  (e.g., to $\ell^p_n,\ n=2,3,\ p\geq 2$), which improves and generalizes
    a result  of R. Gariepy and W.D. Pepe (1972). By the same method we also generalize Fu's result to Riemannian manifolds and improve a result of K. Shiohama and M. Tanaka (1996) on distance spheres in Alexandrov spaces.
    \end{abstract}

\maketitle


\markboth{J. Rataj and L.~Zaj\'{\i}\v{c}ek}{Distance spheres}

\section{Introduction}

    Let  $X$ be a metric space and $F \subset X$  a closed set. We will study level sets of the distance function $S_r(F):= 
    \{ x \in \R^n:\ \dist(x,F) = r\}$, $r>0$. We will call these sets  (following \cite{ST}) {\it distance spheres}; they are sometimes called also  
      $r$-boundaries of $F$ (see \cite{Fe}). There is a number of articles that investigate properties of distance spheres. R. Gariepy and W.D. Pepe  \cite{GP} studied distance spheres in a Minkowski space (= finite dimensional Banach space) $X$. S. Ferry \cite{Fe} proved that  if $X= \R^2$ or $X=\R^3$, then, for almost all $r>0$,
 the distance sphere ($r$-boundary) $S_r(F)$ is a topological $(n-1)$-dimensional
  manifold. This result was improved by J.H.G. Fu \cite{Fu}, who proved that these topological manifolds are very nice: they are semiconcave surfaces. Moreover, he proved that, for $n=2$, the above property of $S_r(F)$ is valid for all
   $r>0$ except a relatively closed set $N\subset (0,\infty)$ with $\Ha^{1/2}(N) = 0$. 
   
   We observe that Ferry's result is an easy consequence (see Theorem \ref{abstd2})
   of the only fact that the distance function  $d(x)=  \dist(x,F)$ is locally DC (i.e.,  a difference of two convex functions)  and has no stationary point
   in $\R^n \setminus F$. Using this observation, we show that Ferry's  (and even Fu's) result extends to sufficiently smooth normed linear spaces $X$ with $\dim X \in \{2,3\}$  (e.g., to $\ell^p_n,\ n=2,3,\ p\geq 2$), which improves and generalizes
    a result  of R. Gariepy and W.D. Pepe \cite{GP}.
    
    If $X$ is a Riemannian manifold, then it is well-known (see \cite[p.~34]{Fu00} or \cite{MM}) that $d(x)=  \dist(x,F)$ is locally semiconcave
     (and therefore also locally DC) in arbitratry local coordinates. So, we can apply Theorem \ref{abstd2} and obtain Ferry's (and even Fu's) results.

     K. Shiohama and M. Tanaka  \cite{ST} studied distance spheres of compact subsets $F$ of a connected two-dimensional complete Alexandrov space without boundary $X$. They proved that then, for almost all $r>0$, the distance sphere $S_r(F)$ is rectifiable and consists of a disjoint union of finitely many simply closed curves. Using our method and Perelman's DC structure on Alexandrov spaces, we obtain a result, which improves that of \cite{ST}. Namely, we show that $S_r(F)$ is a one-dimensional Lipschitz  manifold  for all
   $r>0$ except a closed set $N\subset [0,\infty)$ with $\Ha^{1/2}(N) = 0$. 
   
   If $X$ is a three-dimensional complete Alexandrov space (possibly with boundary points), then our method gives only
    that, for almost every $r>0$, the set $S_r(F) \cap X^*$ is a two-dimensional Lipschitz manifold, where $X^*$ is the set of all ``Perelman regular'' points (note that $X^*$ is an open dense convex subset of $X$, cf. Section 6). Consequently, if
     $\Ha^{1}(X \setminus X^*) = 0$,   then Ferry's result extends to $X$.  
    
    In all  types of spaces considered above, we obtain also weaker results on distance spheres in $n$-dimensional
     spaces with arbitrary $n$. Namely, we prove that, except a countable set of radii $r$, there exists an
      $(n-1)$-dimensional Lipschitz manifold $A_r \subset S_r(F)$ such that $A_r$ is open and dense in $S_r(F)$ and 
       $\Ha^{n-1}(S_r(F) \setminus A_r)=0$.
(For the density of $A_r$ in $S_r(F)$  in Alexandrov spaces we need that $X=X^*$.)

\section{Preliminaries}
The symbol $B(x,r)$ will denote the open ball
  with center $x$ and radius $r$.
  
\begin{definition}\label{reach}
Let $X$ be a metric space.  
Given a nonempty subset $A\subset X$ and $p\in A$, the {\it reach of $A$ at $p$}, $\reach (A,p)$, is defined as the supremum of all $\ep>0$ such that any point $q\in X$ with $\dist (p,q)<\ep$ has its unique nearest point in $A$. We set $\reach A:=\inf_{p\in A}\reach (A,p)$ and say that {\it $A$ has positive reach} if $\reach A>0$. The set $A$ is said to have {\it locally positive reach} if $\reach (A,p)>0$ for all $p\in A$. 
\end{definition}

\begin{remark}
Sets with positive reach were introduced by Federer \cite{Fed59} in the Euclidean space and by Kleinjohann \cite{Kj80} in Riemannian manifolds. Note that if $A$ is compact then $\reach A>0$ whenever $A$ has locally positive reach. It follows from the obvious fact that $\reach (A,p)$ 
 depends continuously on $p\in A$.
\end{remark}

\begin{definition}\label{lipman}(cf.\ \cite{FuCur}, p.~622)\  
We say that a metric space $X$ is an {\it $m$-dimensional Lipschitz manifold} if for every $a\in X$ there exists an open neighbourhood $U$ of $a$
 and a bilipschitz homeomorphism  $\vf$ of $U$ onto an open subset of $\R^m$.
 \end{definition}
  
Let $X$ be a normed linear space and let $f$ be a real function defined on an open set $G \subset X$.

   The {\it directional derivative} and the {\it one-sided  directional derivative} 
    of $f$ at $a\in G$ in the direction $v \in X$ are defined by 
   $$ f'(a,v) := \lim_{t \to 0} \frac{f(a+tv)-f(a)}{t} \ \ \text{and}\ \ \ f'_+(a,v) := \lim_{t \to 0+} \frac{f(a+tv)-f(a)}{t}.$$
   
  Now suppose that $f$ is locally Lipschitz on $G$. Then 
    $$f^0(a,v) : = \limsup_{z \to a, t \to 0+}\ \frac{f(z+tv)-f(z)}{t}$$
   is the {\it Clarke derivative} of $f$ at $a\in G$ in the direction $v\in X$ and
   $$ \partial^Cf(a) := \{x^* \in X^*:\ \langle x^*,v\rangle \leq  f^0(a,v)\ \ \text{for all}\ \ v \in X\}$$
   is the {\it Clarke subdifferential} of $f$ at $a$.
   We shall use the following terminology (see  \cite{Fu}).
   \begin{definition}\label{critreg}
   Let $f$ be a locally Lipschitz function on an open subset $G$ of a normed linear space. Then we say 
    that $a$ is a {\it regular point} of $f$ if $0 \notin \partial^Cf(a)$. If $0 \in \partial^Cf(a)$,
   we say that $a$ is a {\it critical point} of $f$. The set of all critical points of $f$ will be denoted by 
    $\Crit(f)$. By the set of {\it critical values} of $f$ we mean the set $\cv(f) := f(\Crit(f))$.
   \end{definition}
    
 We will need the following easy lemma. Because of a lack of a reference we supply the obvious proof.   
\begin{lemma}\label{charcrit}
Let $f$ be a locally Lipschitz function on an open set $G \subset \R^n$ and $a \in G$. Then the following conditions are equivalent.
\begin{enumerate}
\item $a \notin \Crit(f).$
\item
 There exist $\delta >0$, $\varepsilon >0$, and $v \in \R^n$ such that
$$  \frac{f(x+tv) - d(x)}{t} < - \varepsilon, \ \ \ \text{whenever}\ \ \ t>0, x \in U_{\delta}(a), x+tv \in U_{\delta}(a).$$
\end{enumerate}
\end{lemma}
\begin{proof}
The condition (i) (i.e., $0 \notin \partial^Cf(a)$) holds if and only if there exists $v \in \R^n$ such that
 $f^0(a,v) = \limsup_{z \to a, t \to 0+}\ \frac{f(z+tv)-f(z)}{t} < 0$. It is easy to see that the last
  condition is equivalent to (ii).
  \end{proof}
   
   Lemma \ref{charcrit} immediately implies the well-known fact that
   \begin{equation}\label{crcl}
   \Crit(f) \ \ \text{is closed}\ \ \text{in}\ \ G.
   \end{equation}

 If $f$ is a real function on a normed linear space $X$, then
      the symbol $f'(a)$ stands for the (Fr\' echet) derivative of $f$ at $a\in X$. If $f'(a)$ exists and
      $$ \lim_{x,y \to a, x\neq y} \ \frac{f(y)-f(x)- f'(a)(y-x)}{\|y-x\|} =0,$$
       then we say that $f$ is {\em strictly differentiable} at $a$ (cf. \cite[p.~19]{Mor}).

   Lemma \ref{charcrit} easily implies the well-known fact (see, e.g., \cite[Proposition 2.2.4]{Cl}, where a weaker notion of strict differentiability is used) that
   \begin{equation}\label{strcrit}
   \text{If}\ \  f'(a) \neq 0\ \ \text{ and}\ \  f\ \  \text{is strictly differentiable at}\ \  a,\ \ \text{ then}\ \  a \notin \Crit(f).
   \end{equation}

\begin{definition}\label{dcfce}
Let $C$ be a nonempty convex set in a real normed linear space $X$. A
function $f\colon C\to\R$ is called {\em DC}\ (or  d.c., or ``delta-convex'')
if it can be represented as a difference of two continuous convex
functions on $C$.

If  $Y$ is a finite-dimensional normed linear space, then a mapping $F\colon C\to Y$ is called {\em DC}, if
 $y^*\circ F$ is a DC\ function on $C$ for each linear functional $y^* \in Y^*$.
 \end{definition}

 \begin{remark}\label{podc}
 \begin{enumerate}
 \item To prove that $F$ is DC, it is clearly sufficient to show that $y^*\circ F$ is DC for each $y^*$ from a basis of $Y^*$.
 \item Each DC\ mapping is clearly locally Lipschitz.
 \end{enumerate}
\end{remark}
We will need the following  properties of DC functions and mappings. 
 
 \begin{lemma}\label{zakldc}
 Let $X,Y,Z$ be finite-dimensional normed linear spaces, let $C\subset X$ be a nonempty convex set, and
   $U \subset X$ and $V \subset Y$ open sets.
 \begin{enumerate}
 \item[(a)] 
 If the derivative of a function $f$ on $C$ is  Lipchitz, then  $f$ is DC. In particular, each affine mapping is DC.
 \item[(b)] 
 Let a mapping $F: U \to Y$ be locally  DC, $F(U) \subset V$, and let   $G:V\to Z$ be  locally DC.
  Then $G \circ F$ is locally DC\ on $U$.
 \item[(c)]
 If mappings  $F: U \to Y$ and $G: U \to Y$ are locally DC, then $F+G$ is also locally DC.
 \item[(d)]
 Let $n \in \{1,2\}$, $\dim X =n$, and let $f$ be a locally DC function on $U$. Let $S:= \{x\in U:\ f'(x)=0\}$ be the set
  of all stationary points of $f$. Then  $\Ha^{n/2}(f(S)) = 0$.
  \end{enumerate}
 \end{lemma}
 The proofs  of (a)--(c) can be found in \cite{VeZa}. Let us note that (a) was at first proved in \cite{A1}, and (b) in \cite{Ha}.
 
 The Morse-Sard theorem (d) was for $n=2$ published by Landis \cite{Lan} with a sketch of the proof. A detailed proof
  based on the modern theory of BV  functions can be found in \cite[Corollary~4.5]{PZ}. The easier case $n=1$ is proved in \cite{Pavlica}.

 An important subclass of the class of DC functions is formed by semiconcave functions (cf. \cite{CaSi}).
 
 \begin{definition}\label{semic}
 Let $H$ be a unitary space. A real function $u$ on an open convex set $C \subset H$ is called {\it semiconcave} (with a semiconcavity constant $c$) if the function
 $$ g(x):= u(x) - (c/2) \|x\|^2$$
  is concave on $C$. 
  
  A function $u$ on an open set $G \subset H$ is called {\it locally semiconcave} if, for each $x \in G$, there exists $\delta>0$ such that $u$ is semiconcave on $B(x,\delta)$. A function $g$ on $G$ is called {\it locally semiconvex}
  if $-g$ is locally semiconcave.
 \end{definition}
 
 \begin{remark}\label{cnekarovder}
 \begin{enumerate}
 \item
  $g$ is locally semiconvex on $G$ if and only if, for each $x \in G$, there exists $\delta>0$ and a $C^{\infty}$ smooth
   function $s$ on $B(x,\delta)$ such that the function $g+s$ is convex on $B(x,\delta)$. It follows, e.g., from
    \cite[Proposition 1.1.3]{CaSi} applied to $u:= -g$.
    \item
     If $g$ is locally semiconvex on $G$, $a\in G$ and $v \in H$, then $g^0(a,v) = g'_+(a,v)$. It follows, e.g., from
    \cite[Theorem 3.2.1]{CaSi} applied to $u:= -g$.
  \end{enumerate}

 \end{remark}

\begin{definition}\label{surfldc}
Let $H$ be an $n$-dimensional unitary space  and $1 \leq k <n$.

\begin{enumerate}
\item
 We say that a set $\emptyset \neq M \subset H$ is a {\it $k$-dimensional Lipschitz surface} (resp. a  {\it $k$-dimensional DC surface}) in $H$, if  for each $x \in M$ there exists a $k$-dimensional linear space $Q \subset H$, an open neighbourhood $W$ of  $x$, a set $G \subset Q$ open in $Q$ and a Lipschitz (resp. locally DC) mapping $h: G \to Q^{\perp}$ such that
$$  M \cap W = \{u + h(u): u \in G\}.$$
\item
We say that a set $\emptyset \neq M \subset H$ is an {\it $(n-1)$-dimensional semiconcave surface} in $H$,
 if  for each $x \in M$ there exists an $(n-1)$-dimensional linear space $Q \subset H$, an open neighbourhood $W$ of  $x$, a set $G \subset Q$ open in $Q$, a vector $0 \neq v \in Q^{\perp}$  and a  locally semiconcave function
 $s: G \to \R$ such that
$$  M \cap W = \{u + s(u)v : u \in G\}.$$
\end{enumerate}
By  a $0$-dimensional Lipschitz (resp. DC, resp. semiconcave) surface we mean a singleton.
 \end{definition}

\begin{remark}\label{surf}
Obviously, each  $k$-dimensional DC surface in $H$ is a $k$-di\-men\-sio\-nal Lipschitz surface in $H$, and each 
$(n-1)$-dimensional semiconcave surface in $H$ is an $(n-1)$-dimensional DC surface in $H$.
\end{remark}

Using the preceeding definition, we can formulate some  versions of known implicit function theorems for Lipschitz, DC and semiconcave functions in a concise form.

\begin{proposition}\label{implic}
Let $f$ be a locally Lipschitz function on an open set $G \subset \R^n$, and let $a \in G \setminus \Crit(f)$. Denote
 $M:= \{x \in G:\ f(x) = f(a)\}$. Then there exists $\delta >0$ such that:
 \begin{enumerate}
 \item
  $M \cap B(a,\delta)$ is an $(n-1)$-dimensional Lipschitz surface in $\R^n$.
  \item
  If $f$ is locally DC, then  $M \cap B(a,\delta)$ is an $(n-1)$-dimensional DC surface in $\R^n$.
  \item
  If $f$ is locally semiconcave, then  $M \cap B(a,\delta)$ is an $(n-1)$-dimensional semiconcave surface in $\R^n$.
   Moreover,  $ \reach(\{x \in G:\ f(x) \geq  f(a)\},a)>0$.
\end{enumerate}
\end{proposition}
\begin{proof}
The statement (i) is an obvious reformulation of \cite[Theorem 3.1]{Fu}, which is an easy consequence of Clarke's
 implicit function theorem. 
 
We will show that the statement (ii) is an easy consequence of \cite[Proposition 5.9]{VeZa}. To this end, choose
 $\delta > 0$, $\varepsilon >0 $ and $v \in \R^n$ as in Lemma \ref{charcrit}(ii). Let $Y$ be the linear span of $\{v\}$ and $X:= Y^{\perp}$. Identifying by the standard way $\R^n$ with $X \times Y$ and, by linear isometries,
   $X$ with $\R^{n-1}$ and $Y$ with $\R$, we can apply  \cite[Proposition 5.9]{VeZa} (with $G:= f$) and obtain 
    the assertion of (ii). Indeed, the fact that $\partial_2f(a)$ contains surjective linear mappings $\R \to \R$ only
     is an easy consequence of the choice of $v$ (see Lemma \ref{charcrit}(ii)). Note  also that Lemma \ref{charcrit}(ii)
      immediately implies the local validity of the inequality $\|G(x,y) - G(x,\overline{y})\| \geq c \|y - \overline{y}\|$ (with $c:= \varepsilon$) which is claimed without a proof in the proof of \cite[Proposition 5.9]{VeZa}.

The first part of the assertion (iii) follows immediately from \cite[Theorem 3.3]{Fu}.
The second part follows easily from the proof of \cite[Corollary 3.4]{Fu} or from \cite[Theorem]{Bangert82} (see (Ban) after Lemma \ref{LRiem}).    
\end{proof}

 \begin{lemma}\label{aznapl}
 Let $X,Y$ be finite-dimensional unitary spaces with $\dim X = n>0$ and $\dim Y = m>0$.
  Let   $G \subset X$ be an  open  set, and  $f: G \to Y$  a locally  DC\ mapping. Then
     there exists a sequence $(T_i)$ of $(n-1)$-dimensional DC surfaces in $X$ such that $f$ is strictly differentiable at each
      point of   $G \setminus \bigcup_{i=1}^{\infty} T_i$.
   \end{lemma}
   \begin{proof}
   We can suppose that $X = \R^n$ and $Y= \R^m$.
   First suppose $n>1$. 
   Using separability of $X$, we can clearly suppose that $G$ is convex and $f$ is DC on $G$.
   Let  $f = (\alpha_1 - \beta_1, \ldots,\alpha_m - \beta_m) $, where all $\alpha_j$ and $\beta_j$ are convex functions.
   By \cite{Zaj}, for each $j$ we can find a sequence $T_k^j, k \in \N,$  of $(n-1)$-dimensional DC\ surfaces in $G$ such that both $\alpha_j$ and $\beta_j$ are differentiable at each point of $ D_j:= G \setminus \bigcup_{k=1}^{\infty} T_k^j$. Since each convex function is strictly differentiable at each point at which it is
    (Fr\' echet) differentiable (see, e.g., \cite[Proposition 3.8]{VeZa} for a proof of this well-known fact), we conclude that each $f_j:= \alpha_j - \beta_j$ is strictly differentiable at each point
     of $D_j$. Since strict differentiablity of $f$ clearly follows from strict differentiability of all $f_j$'s, the proof is finished after  ordering all the sets 
      $T^j_k$, $k \in \N$, $j=1,\dots,m$, to a sequence $(T_i)$.
      
      If $n=1$, we proceed quite similarly, using the well-known fact that a convex function on an open interval is differentiable except a countable set. 
   \end{proof}

\section{Critical values and level sets of DC functions}

\begin{lemma}\label{critdc}
Let $f$, $g$ be convex functions on an open convex  set $C \subset \R^n$, and let $d:= f-g$.
 Assume that the directional derivatives
 $f'(x,v)$, $g'(x,v)$  exist
 for some $x\in C$ and $v\in \R^n$, and  that $f'(x,v) \neq g'(x,v)$. Then  $x \notin \Crit(d)$.
\end{lemma}
\begin{proof}
We can suppose that $f'(x,v) < g'(x,v)$ (otherwise consider $-v$ instead of $v$). Since $f$ is convex, we have
  (cf.\ Remark~\ref{cnekarovder}(ii)) 
 $$ f'(x,v) =  f^0(x,v) = \limsup_{y \to x, t \to 0+} \frac{f(y+tv) - f(y)}{t}$$
 and
 $$ -f'(x,v) = f'(x,-v) =f^0(x,-v)= \limsup_{y \to x, t \to 0+} \frac{f(y-tv) - f(y)}{t} $$
 $$= -\liminf_{y \to x, t \to 0+} \frac{ f(y)- f(y-tv)}{t}=   -\liminf_{z \to x, t \to 0+}\frac{f(z+tv) - f(z)}{t}.$$
 Consequently  
  $$ f'(x,v) = \lim_{y \to x, t \to 0+} \frac{f(y+tv) - f(y)}{t}.$$
  Using this also for the convex function $g$, we obtain
  $$ \lim_{y \to x, t \to 0+} \frac{d(y+tv) - d(y)}{t} = f'(x,v) -g'(x,v) < 0.$$
   Thus there exist $\varepsilon >0$ and $\delta >0$ as in  (ii) of Lemma \ref{charcrit}.
   \end{proof}

\begin{lemma}\label{main}
Let $X$ be an $n$-dimensional unitary space and let $k \in \{1,2\}$ with $k<n$. 
Let $C \subset X$ be an open convex  set, and let $d$ be a DC function on $C$. Let   $P \subset X$  be a $k$-dimensional DC surface.
 Then 
 $$ \Ha^{k/2} \big( d(P \cap \Crit(d))\big) =0.$$
\end{lemma}
\begin{proof}
Let $d= f-g$, where $f$, $g$ are convex functions on $C$. 

(i)\ First suppose $k=1$. Using separability of $X$, we can clearly suppose that $P=\{t + h(t): t \in G\}$, where
  $G $ is  a relatively open subset of a one-dimensional linear space $V \subset X$, and  $h: G \to V^{\perp}$ 
  is a locally DC mapping.
 Set $\vf(t):= t + h(t),\ t \in G$. Then $\vf$ is locally DC on $G$ (Lemma \ref{zakldc}(a),(c)), and so also $f \circ \vf$, $g \circ \vf$ and $d\circ \vf$ are locally DC on $G$ (Lemma \ref{zakldc}(b)). So, by  Lemma \ref{aznapl} there exists  a countable set  
(countable union of $0$-dimensional DC surfaces) $A \subset G$
such that 
$\vf'(t)$, $(f \circ \vf)'(t)$ and  $(g \circ \vf)'(t)$  exist for each $t \notin A$. Set $B:= \{x \in G \setminus A:\ (f \circ \vf)'(t) =(g \circ \vf)'(t)\}$.   For each $t \in B$, we have $(d \circ \vf)'(t)=0$, and consequently  $\Ha^{1/2}(d \circ \vf(B)) = 0$ by Lemma \ref{zakldc}(d).
 Set
 $$N:= (d\circ \vf)(A) \cup (d\circ \vf)(B)= d\big(\vf(A) \cup \vf(B)\big).$$ 
 Since clearly   $\Ha^{1/2}(N) = 0$, it is sufficient to prove
 \begin{equation}\label{podmn}
 P \cap \Crit(d) \subset \vf(A) \cup \vf(B).
 \end{equation}
 To this end, suppose that $x \in P \setminus (\vf(A) \cup \vf(B))$.
  Then $x = \vf(t)$ for some
  $t \in G \setminus (A \cup B)$. So, $\vf'(t)$ exists and $(f \circ \vf)'(t) \neq (g \circ \vf)'(t)$. 
  So we can choose $u \in V$ such that $(f \circ \vf)'(t,u)\neq (g \circ \vf)'(t,u)$. Set  $v:= \vf'(t)(u)$.
  Since  $f'_+(x,v)$, $f'_+(x,-v)$ exist and $f$ is locally Lipschitz, we conclude
  (see, e.g. \cite[Proposition 3.6(i)]{Sh}) that
   $f'_+(x,v) = (f\circ\vf)'(t,u)$ and similarly
    $$ \ f'_+(x,-v) = (f\circ\vf)'(t,-u)=-(f\circ\vf)'(t,u)= -f'_+(x,v).$$
   Consequently $f'(x,v)$ exists. Similarly we obtain that  $g'(x,v)$ exists. Thus
    $f'(x,v)= (f \circ \vf)'(t,u)\neq (g \circ \vf)'(t,u) = g'(x,v)$, and consequently $x \notin \Crit(d)$ by Lemma \ref{critdc}. So \eqref{podmn}
     holds.

 (ii)\ Let now $k=2$.
  Using separability of $X$, we can clearly suppose that $P=\{t + h(t): t \in G\}$, where
    $G $ is  a relatively open subset of a two-dimensional linear space $V \subset X$, and  $h: G \to V^{\perp}$ 
  is a locally DC mapping.
 Set $\vf(t):= t + h(t),\ t \in G$. Then $\vf$ is locally DC on $G$, and so also $f \circ \vf$, $g \circ \vf$ and $d\circ \vf$ are locally DC on $G$. So, by  Lemma \ref{aznapl} 
 there exists a sequence $(P_i)_{i=1}^{\infty}$ of one-dimensional DC surfaces in $V$ such that $\vf'(t)$,  $(f \circ \vf)'(t)$ and  $(g \circ \vf)'(t)$  exist for each $x \in G \setminus A$, where $A:= \bigcup_{i=1}^{\infty} P_i$. Using separability of $V$, we can suppose that each $P_i$ is of the form
  $P_i = \{s + g_i(s):\ s \in H_i\}$, where $H_i$ is a relatively open subset of a one-dimensional linear space $W_i \subset V$ and 
   $g_i: H_i \to (W_i^{\perp} \cap V)$ is a locally DC mapping. Put $Q_i:= \vf(P_i) = \{s + g_i(s) + h(s+g_i(s)):\ s \in H_i\}$. Since 
   $\psi(s):= g_i(s) + h(s+g_i(s))$, $ s \in H_i$,  is a locally DC mapping $\psi: H_i \to W_i^{\perp}$ (Lemma \ref{zakldc}(b),(c)), we 
  obtain that each $Q_i$ is a one-dimensional DC surface in $X$.

  Set  $N_1 := d(\bigcup_{i=1}^{\infty} Q_i \cap \Crit(d))$. By part (i) of the proof, $\Ha^{1/2}(N_1) = 0$.
  Set $B:= \{ t \in G\setminus  \bigcup_{i=1}^{\infty} P_i:  \ (f \circ \vf)'(t)= (g \circ \vf)'(t)\}.$
    For each $t \in B$, we have $(d \circ \vf)'(t)=0$, and consequently  $\Ha^{1}(d \circ \vf(B)) = 0$ by the Morse-Sard theorem for DC functions (Lemma \ref{zakldc}(d)) on $\R^2$.
  Set $$N:= N_1 \cup (d\circ \vf)(B)=   d\big((\vf(A)\cap \Crit(d)) \cup \vf(B)\big).$$ 
   Since clearly   $\Ha^{1}(N) = 0$, it is sufficient to prove
 \begin{equation}\label{podmn2}
 P \cap \Crit(d) \subset \vf(A) \cup \vf(B).
 \end{equation}
The proof of \eqref{podmn2} can be done literally as the proof of \eqref{podmn}.
 \end{proof}

 \begin{proposition}\label{abstd}
 Let  $n \in \{2,3\}$ and let $d$ be a locally DC function on an open set $G \subset \R^n$. Suppose that $d$
  has no stationary point.  Let $\cv(d)= d(\Crit(d))$ be the set of critical values of $d$. Then 
   $\Ha^{(n-1)/2}(\cv(d)) =0$. 
\end{proposition}
 \begin{proof}
We can and will assume that $G$ is convex. 
 By Lemma \ref{aznapl} there exists a sequence $(P_i)_{i=1}^{\infty}$ of $(n-1)$-dimensional DC surfaces in $\R^n$ such that
  $d$ is strictly differentiable (and $d'(x) \neq 0$ by the assumptions) at each $x \in G \setminus \bigcup_{i=1}^{\infty} P_i$. So, using \eqref{strcrit}, we obtain $ \Crit(d) \subset \bigcup_{i=1}^{\infty} P_i$. Applying
   Lemma \ref{main} for each $i \in \N$, we obtain  that    $\Ha^{(n-1)/2}(\Crit(d) \cap P_i) =0$ for each $i$, and therefore $\Ha^{(n-1)/2}(\cv(d)) =0$. 
 \end{proof}

\begin{theorem}\label{abstd2}
 Let  $n \in \{2,3\}$ and let $d$ be a locally DC function  on an open set $G \subset \R^n$.  Suppose that $d$
  has no stationary point. Then there exists a set $N \subset \R$ with $\Ha^{(n-1)/2}(N) =0$ such that, for every $r \in d(G) \setminus N$,
  the set $d^{-1}(r)$ is  an $(n-1)$-dimensional DC  surface. If $d$ is even locally semiconcave, we can also assert that $d^{-1}(r)$ is  an $(n-1)$-dimensional semiconcave  surface and the set $\{x \in G:\ d(x) \geq r\}$ has locally
   positive reach.
   
   Moreover,  $N$ can be chosen so that $N= d(C)$, where $C$ is a closed set in $G$. Namely, we can put
    $N:=  \cv(d)= d(\Crit(d))$. 
     \end{theorem}
\begin{proof}
Set $C:= \Crit(d)$ and $N:= d(C)$. Then $C$ is closed in $G$ by \eqref{crcl}. Proposition~\ref{abstd} yields $\Ha^{(n-1)/2}(N)=0$. Let $r \in d(G)\setminus N$. Applying Proposition \ref{implic}
 (with $f:= d$) to each point $a \in d^{-1}(r)$, we easily obtain that the sets  $ d^{-1}(r)$ and $\{x \in G:\ d(x) \geq r\}$ have the desired properties.
\end{proof}

The following weaker result holds for all dimensions $n$.

\begin{theorem}\label{kazden}
Let $d$ be a locally DC (resp. locally semiconcave) function  on an open set $G \subset \R^n$ and assume that $d$ has no stationary point. 
    Then, for all
   $r\in d(G)$, except a countable set, the set  $A_r:= d^{-1}(r) \setminus \Crit(d)$ 
is an $(n-1)$-dimensional DC surface (resp. semiconcave surface) 
  which is open and dense in $ d^{-1}(r)$ and $\Ha^{n-1}(d^{-1}(r) \setminus A_r)=0$.
\end{theorem}

\begin{proof}
If $r\in \R$ and 
 $A_r:= d^{-1}(r) \setminus \Crit(d)$ is nonempty, then it is an 
 $(n-1)$-dimensional DC surface (resp. semiconcave surface) by Proposition \ref{implic}.
 Set
\begin{eqnarray*}
 N_1 &:=& \{r\in \R:\ \Ha^{n-1}(d^{-1}(r) \cap \Crit(d))> 0\},\\
 N_2 &:=& \{r\in \R:\ d\ \ \text{has a local extreme at a point of}\ \ d^{-1}(r)\}
\end{eqnarray*} 
 and  $N:= N_1 \cup N_2$. 
By Lemma \ref{aznapl} and \eqref{strcrit}, $\Crit(d)$ can be covered by countably many  $(n-1)$-dimensional DC surfaces and therefore $\Ha^{n-1}$ is  $\sigma$-finite on  $\Crit(d)$. Thus $N_1$ is countable. 
 It is well-known (and easy to prove) that the set of (possibly non-strict) extremal values of a real function on a separable metric space $Y$ is countable. (The proof for $Y=\R$  \cite[p.~43]{RS} easily generalizes to general  $Y$.) Thus $N_2$, and so also $N$, is countable. Note that $A_r$ is open in $d^{-1}(r)$ by \eqref{crcl}. To conclude the proof, it is sufficient to prove that $A_r = d^{-1}(r) \setminus \Crit(f)$ is dense in $d^{-1}(r)$ for each $r \in (0,\infty) \setminus N$.
  
   To this end, suppose on the contrary
   that there exist $r \in \R \setminus N$ and a point $a \in d^{-1}(r) \setminus \overline{A_r}$.
    Choose a convex open neighbourhood $U \subset G$ of $a$ such that  $U \cap A_r = \emptyset$.
     Since $r \notin N_2$, we can choose points $b, c \in U$ such that $d(b) >r$ and $d(c) <r$.
      Set  $W:= (c-b)^{\perp}$   and
       $B_{\delta}:= \{w \in W:\ \|w\| < \delta\}$. Choose $\delta >0$ so small that $d(x) > r$ for each
        $x \in b + B_{\delta}$ and $d(y) <r$ for each $y \in c + B_{\delta}$. Then, for each $w \in  B_{\delta}$,
         we can clearly find a point $z_w$ in the segment $S_w:= \{b+w+ t(c-b):\ t \in [0,1]\}$ with
          $d(z_w)=r$. Denote $Z:= \{z_w:\ w  \in  B_{\delta}\}$. Since the mapping $z_w \mapsto w$ is Lipschitz with constant $1$ on
           $Z$ and $\Ha^{n-1}(B_{\delta})>0$, we obtain $\Ha^{n-1}(U \cap d^{-1}(r)) \geq  \Ha^{n-1}(Z)>0$.  Since $U \cap d^{-1}(r) \subset  \Crit(d)$, we obtain a contradiction with $r \notin N_1$.
    \end{proof}

 \section{Minkowski spaces}

Let $X$ be a Minkowski space (= finite dimensional Banach space). R. Gariepy and W.D. Pepe \cite{GP} proved the following results.
\medskip

(GP1)\ \ {\it If $\dim X = n$, the norm of $X$ is strictly convex or differentiable and $F \subset X$ is a closed set, then, for almost
 every  $r>0$, the distance sphere $S_r(F)$ is either empty, or there there exists an $(n-1)$-dimensional Lipschitz manifold $A_r \subset S_r(F)$
  such that $A_r$ is open in $S_r(F)$ and $\Ha^{n-1}(S_r(F) \setminus A_r)=0$ .}
  \medskip
  
  (GP2)\ \ {\it If $\dim X = 2$, the norm of $X$ is twice differentiable with bounded second derivative on the unit sphere 
   and $F \subset X$ is a closed set, then, for almost
 every  $r>0$, the distance sphere $S_r(F)$ is either empty, or a one-dimensional Lipschitz manifold.}
 \medskip
 
 S. Ferry \cite{Fe} proved that {\it if $X= \R^n$ with $n\in\{ 2,3\}$ then, for almost all $r>0$,
 the  distance sphere  $S_r(F)$ is  either empty or a topological $(n-1)$-dimensional
  manifold.} He also showed that this result does not hold in $\R^n$ for $n\geq 4$.
\medskip  
  
  J.H.G. Fu  \cite{Fu} essentially (cf. Remark \ref{Fu}) proved the following stronger result.
  \medskip
  
  (Fu)\ \ {\it Let $X=\R^n$, $n\in\{ 2,3\}$, and $F \subset X$ be a nonempty compact set. Then there exists a compact set
  $N \subset [0,\infty)$ with $\Ha^{(n-1)/2} (N) = 0$ such that, for every $r \in (0,\infty) \setminus N$,
 the  distance sphere  $S_r(F)$ is  an $(n-1)$-dimensional semiconcave 
  surface and $\overline{\{x:\ \dist(x,F) >r\}}$ has positive reach.}

  \medskip                                                
  
\begin{remark}\label{Fu}
Fu did not consider distance spheres $S_r(F)$ but the sets $S^*_r(F):= \partial B_r(F)$, where $B_r(F):= \{x\in X:\ \dist (x,F) \leq r\}$.
However, this difference is not essential, since the set $\{ r>0:\ S_r(F) \neq S^*_r(F)\}$ is countable (even for any $n \in \N$ and any nonempty closed $F \subset \R^n$).

Fu formulated his result in a formally different way:\ he asserted that, for every $r \in (0,\infty) \setminus N$,
 $S^*_r(F)$ is a Lipschitz $(n-1)$-dimensional
  manifold  and  $\overline{X \setminus B_r}$ is a set of positive reach. However, the proofs of \cite{Fu} give that,
  for every $r \in (0,\infty) \setminus N$, the set 
 $S^*_r(F)$ is an $(n-1)$-dimensional semiconcave 
  surface and $S^*_r(F)=S_r(F)$.
\end{remark}

Our first result generalizes in a sense (Fu) to sufficiently smooth normed linear spaces and generalizes and improves
 (GP2).
 
 \begin{theorem}\label{lipder}
  Let  $n \in \{2,3\}$ and let $(X, \|\cdot\|)$ be an $n$-dimensional  normed linear space  such that the derivative of the norm 
  $\|\cdot\|$ is Lipschitz
  on the unit sphere (e.g., $X=\ell^p_n$, $p\geq 2$).  Let $F \subset X$ be a nonempty closed  set and denote 
  $S_r(F):= \{x\in X:\ \dist_{\|\cdot\|}(x,F) = r\}$,  $F_{\geq r}:= \{x\in X:\ \dist_{\|\cdot\|}(x,F) \geq r\}$. 
  
  Then there exists a  set
  $N \subset (0,\infty)$ with $\Ha^{(n-1)/2} (N) = 0$ such that, for every $r \in (0,\infty) \setminus N$:
  \begin{enumerate}
  \item
 The  distance sphere  $S_r(F)$  is either empty or  an  $(n-1)$-dimensional Lipschitz manifold in  $(X, \|\cdot\|)$.
 \item
 If $\|\cdot\|_H$ is an arbitrary (equivalent) Hilbert norm on $X$, then $S_r(F)$  is either empty or  an  $(n-1)$-dimensional 
  semiconcave
  surface in $(X, \|\cdot\|_H)$ and  $F_{\geq r}$ has locally positive reach  in 
   $(X, \|\cdot\|_H)$.     
   \item
   If $F$ is compact then $N$ is closed in $(0,\infty)$ and   $F_{\geq r}$ has positive reach  in 
   $(X, \|\cdot\|_H)$.    
   \end{enumerate}
 \end{theorem}
\begin{proof}
First observe that the norm of $X=\ell^p_n$ ($p\geq 2$) has the assumed property; see e.g. \cite[proof of Corollary 1.2, p.~187]{DGZ}. Further observe that (ii)
 immediately implies (i).
 
 To prove (ii), we will need that the distance function $g(x):= \dist_{\|\cdot\|}(x,F)$ is locally semiconcave on $G:=X \setminus F \subset (X, \|\cdot\|_H)$. It follows from the proof of \cite[Theorem 5]{Zadi}, where it is shown that, for each  $x_0 \in G$, the function $g$
 is semiconcave (in $(X, \|\cdot\|_H)$) on the ball $\{x \in X:\ \|x-x_0\| < g(x_0)/2\}$ (although \cite[Theorem 5]{Zadi} only asserts that $g$ is locally DC). Further, no point $x_0 \in G$ is a stationary point of  $g$. Indeed, let $y\in F$ be a point with
  $\|y - x_0\|= g(x_0)$. Since clearly $g(x_0+t(y-x_0)) - g(x_0) = -t\|y-x_0\|$ if $0<t < 1$, we see that $x_0$ is not a stationary point of  $g$.  Now choose an arbitrary linear isometry $L:\,(X, \|\cdot\|_H)\to\R^n$. Applying
  Theorem \ref{abstd2} to the function $d:= g \circ L^{-1}$, we obtain a set $N \subset (0,\infty)$ with the desired properties.
  
  Now suppose that $F$ is compact. Then we will use the fact that, by  Theorem \ref{abstd2}, $N$ can be chosen so that $N=g(C)$, where 
   $C$ is a closed set in $G$. For each $0<a<b<\infty$, we have  that $N \cap [a,b] = g(C \cap \{x \in X: \dist_{\|\cdot\|}(x,F) \in [a,b]\})$ is compact, since $\{x \in X: \dist_{\|\cdot\|}(x,F) \in [a,b]\}\subset G$ is compact and $g$ is continuous. Therefore $N$ is closed in $(0,\infty)$. 
   Finally, choose $\rho>0$ such that $\|x\|_H \leq \rho$ for each  $x \in X\setminus F_{\geq r}$, and observe that 
$$\reach F_{\geq r}=\min\Big\{\inf_{p \in F_{\geq r},\, \|p\|_H \leq \rho+1}\reach(F_{\geq r},p),
\inf_{\|p\|_H> \rho+1}\reach(F_{\geq r},p)\Big\}.$$
The first infimum is positive since $\reach(F_{\geq r},\cdot)$ is continuous and
positive, and $\{p \in F_{\geq r}:\, \|p\|_H \leq \rho+1     \}$ is compact. The second infimum is clearly
greater or equal to $1$. Thus, $\reach F_{\geq r}>0$ and the proof of (iii) is over.
  \end{proof}

  \begin{remark}\label{dokoncedc}
  The property (ii) immediately implies that, if $r \in (0,\infty) \setminus N$, then $S_r(F)$  is either empty or an $(n-1)$-dimensional DC
  surface in $(X, \|\cdot\|)$, if we define this notion in normed spaces in a natural way (as in \cite{Zajploch}).
   \end{remark}

The following result considerably improves (GP1) (since the exceptional set is countable and $A_r$ is  dense in $S_r(F)$), but only in sufficiently smooth normed linear spaces. It seems to be new also in Euclidean spaces.

\begin{theorem}\label{spoc}
Let $X$ be an $n$-dimensional normed linear space ($n\geq 2$) such that the derivative of the norm is Lipschitz
  on the unit sphere (e.g., $X=\ell^p_n$, $p\geq 2$).  Consider on $(X, \|\cdot\|)$ an arbitrary equivalent Hilbert norm $\|\cdot\|_H$. Let $F \subset X$ be a nonempty closed  set. Then,
for all
   $r>0$, except a countable set,
the distance sphere $S_r(F)$ (considered in $(X, \|\cdot\|)$) is either empty, or 
there exists an $(n-1)$-dimensional semiconcave surface $A_r$  in $(X, \|\cdot\|_H)$ 
  such that  $A_r \subset S_r(F)$, $A_r$ is open and dense in $S_r(F)$ and $\Ha^{n-1}(S_r(F) \setminus A_r)=0$ .
\end{theorem}

\begin{proof}
It was shown in the proof of Theorem \ref{lipder} that $g(x):= \dist_{\|\cdot\|}(x,F)$ is locally semiconcave on $G:=X \setminus F \subset (X, \|\cdot\|_H)$
 and has no stationary points in $G$. (Indeed, the proof worked for arbitrary $n\in\N$.) Thus it is sufficient to apply Theorem \ref{kazden}.
\end{proof}

\begin{remark}
Obviously, $A_r$ is an $(n-1)$-dimensional Lipschitz manifold in $(X,\|\cdot\|)$.
\end{remark}

\section{Riemannian manifolds}
Let $M$ be a smooth, complete and connected Riemannian manifold. By $\dist$ we denote the induced inner distance on $M$.
Let $F$ be a nonempty closed subset of $M$, and denote by $d_F:=\dist(\cdot, F)$ the
distance function from $F$.
An {\it $F$-segment} is a unit speed geodesic path $\gamma:[0,a]\to M$ such that
$\gamma(a)\in F$ and $a-t=d_F(\gamma(t))$, $t\in[0,a]$. 
Notice that if $p\in M\setminus F$ then there always exists at least one $F$-segment emanating from $p$.
The following definition is
commonly used in Riemannian geometry, see e.g.\ \cite[\S11.1]{Petersen98} or \cite{Grove}: 

\begin{definition}  \label{Grove}
A point $p\in M\setminus F$ is a {\it critical point of $d_F$} if for any tangent
vector $v\in T_pM$ there exists an $F$-segment $\gamma$ emanating from $p$ and such
that the angle formed by $v$ and $\dot{\gamma}(0)$ is not greater than $\frac\pi 2$. Let
$\Crit(d_F)$ denote the set of all critical points of $d_F$ in $M$. A point $p\in M\setminus F$
is a {\it regular point of} $d_F$ if $p\not\in\Crit(d_F)$. 
\end{definition}

\begin{remark} \label{R-Grove}
Other definitions of critical and regular points of distance functions on
Riemannian manifolds appear in the literature (see, e.g., \cite[p.~34]{Fu00} or
\cite[p.~55]{Bangert82}); fortunately, they are all known to be equivalent. This will be
shown for completeness in Lemma~\ref{LRiem} and follows essentially from the
following observation:
For a point $p\in M\setminus F$, $p\not\in\Crit(d_F)$ if and only if there exists a tangent vector $v\in T_pM$ and $\ep>0$ such that 
$$d_F(c_v(t))\geq d_F(p)+\ep t$$
for all sufficiently small $t>0$, where $c_v$ is the geodesic curve defined on a neighbourhood of $0$ such that $c_v(0)=p$ and $\dot{c}_v(0)=v$, see \cite[p.~360]{Grove}.
\end{remark}

Theorem~\ref{abstd2} yields the following extension of Fu's result to Riemannian
manifolds.

\begin{theorem} \label{Riem}
Let $n\in\{ 2,3\}$ and let $F$ be a nonempty closed subset of a connected complete smooth
$n$-dimensional Riemannian manifold $M$. Then, setting
$N:=d_F(\Crit(d_F))\subset(0,\infty)$, we have $\Ha^{(n-1)/2}(N)=0$ and for all
$r\in d_F(M\setminus F)\setminus N$,
\begin{enumerate}
\item[{\rm (i)}] $S_r(F)$ is an $(n-1)$-dimensional Lipschitz manifold,
\item[{\rm (ii)}] $\{ p\in M:\, d_F(p)\geq r\}$ has locally positive reach.
\end{enumerate}
If, moreover, $F$ is compact then $N$ is relatively closed in
$(0,\infty)$, and $\{ p\in M:\, d_F(r)\geq r\}$ has positive reach for all
$r\in d_F(M\setminus F)\setminus N$.
\end{theorem}

\begin{definition} \label{ext-semiconc}
A function $f$ on $M$ is said to be {\it locally semiconvex} (resp.\ {\it
locally semiconcave}) on an open subset $G\subset M$ if for any chart $(U,\vf)$ with
$U\subset G$, $f\circ\vf^{-1}$ is locally semiconvex (resp.\ locally semiconcave). 
\end{definition}

It is well known that the distance function $d_F$ to a closed subset $F\subset M$ is locally semiconcave on $M\setminus F$, see \cite{MM} (cf.\ also \cite[p.~34]{Fu00}).

Bangert \cite{Bangert79,Bangert82} studied a system ${\cal F}(M)$ of functions on $M$ which turns out to be just the system of locally semiconvex functions (by Remark~\ref{cnekarovder}(i) and the proofs in \cite{Bangert79}). He showed \cite{Bangert79} that the directional derivative $\partial_pf(v)$ of $f\in{\cal F}(M)$ at $p\in M$ exists in any direction $v\in T_pM$, and defined \cite{Bangert82} regular points of $f$ as those points $p\in M$ for which
\begin{equation} \label{Bang-reg}
\exists v\in T_pM:\quad \partial_pf(v)<0.
\end{equation}
This definition (which has in \cite{Bangert82} a formally different, but clearly equivalent form) can be, of course, extended to functions that are locally semiconvex on an open subset of $M$ only.

The following lemma shows that Bangert's terminology is consistent with Definitions~\ref{critreg} and \ref{Grove}.

\begin{lemma} \label{LRiem}
Let $f$ be a locally semiconvex function on an open set $G\subset M$, $p\in G$, and
let $\varphi:U\to\R^n$ be a chart about $p$ with $U\subset G$. Then
\begin{enumerate}
\item[{\rm (i)}] 
Condition \eqref{Bang-reg} holds if and only if $p\not\in\Crit(f\circ\vf^{-1})$.
\item[{\rm (ii)}] The set of points $p\in G$ with property \eqref{Bang-reg} (regular points of $f$ in the sense of Bangert) is open.
\item[{\rm (iii)}] If, in particular, $f=-d_F$ for some closed subset $F\subset M$,
then a point $p\in M\setminus F$ satisfies \eqref{Bang-reg} if and only if $p\not\in\Crit(d_F)$. Moreover, $\Crit(d_F)$ is a closed subset of $M\setminus F$. 
\end{enumerate}
\end{lemma}

\begin{proof}
 From the proof of \cite[(3.1)Satz]{Bangert79}, it follows that
$$\partial_pf=(f\circ\varphi^{-1})_+'(\varphi(p),\cdot)\circ(d\varphi(p)).$$
Since $f\circ\varphi^{-1}$ is locally semiconvex, we have
$$(f\circ\varphi^{-1})'_+(\varphi(p),\cdot)= (f\circ\varphi^{-1})^0(\varphi(p),\cdot)$$
by Remark~\ref{cnekarovder}~(ii), hence, 
$$\partial_pf=(f\circ\varphi^{-1})^0(\varphi(p),\cdot)\circ(d\varphi(p)).$$
Assertion (i) follows then directly from the definitions. Statement (ii) follows from the fact that $\vf(U)\setminus\Crit(f\circ\vf^{-1})$ is open for any chart $\varphi$, and each $\varphi$ is a homeomorphism. Statement (iii) follows from Remark~\ref{R-Grove} and (ii).
\end{proof}

In the proof of Theorem~\ref{Riem} we use the following result due to Bangert (\cite[Theorem]{Bangert82}).
\vskip 2mm

(Ban) \ {\it Let $f$ be locally semiconvex on $M$ and $r\in\R$ be such that
every point $p\in f^{-1}(r)$ is a regular point of $f$. Then $\{ p\in M:\, f(p)\leq r\}$ has locally positive reach.} 
\vskip 2mm

(In fact, Bangert showed a stronger result in \cite{Bangert82}, namely that a weaker
regularity condition is equivalent to the property of locally positive reach.)

We shall also use the fact that each chart $\varphi:U\to\R^n$ of a smooth Riemannian manifold is locally bilipschitz (with respect to the induced inner metric on $M$).
(See, e.g., the proof of \cite[Theorem~3.4]{Petersen98}.)
\vskip 2mm

\noindent{\it Proof of Theorem~\ref{Riem}.} 
Recall that $d_F$ is locally semiconcave on $M\setminus F$. Take a countable atlas $(U_i,\vf_i)$ of $M\setminus F$ and notice that $N=\bigcup_iN_i$ with 
$$N_i:=\cv(d_F\circ\vf_i^{-1})=d_F\circ\vf_i^{-1}(\Crit (d_F\circ\vf_i^{-1})), \quad i\in\N,$$ 
by Lemma~\ref{LRiem}(i) and (iii). 
Further, $d_F\circ\vf_i^{-1}$ has no stationary point. Indeed, for any $p\in U_i$ there exists a unit direction $v\in T_pM$ with directional derivative $\partial_pd_F(v)=-1$ (take $v=\dot{\gamma}(0)$ for an $F$-segment $\gamma$ emanating from $p$), and notice that $(d_F\circ\vf_i^{-1})'_+(\vf_i(p),d\vf_i(p)v)=-1$, hence, $\vf_i(p)$ cannot be a stationary point of $d_F\circ\vf_i^{-1}$.
Hence, $\Ha^{(n-1)/2}(N)=0$ by Proposition~\ref{abstd}.
Since $M$ is complete, it is boundedly compact by the Hopf-Rinow theorem
\cite[Theorem~7.1]{Petersen98} and, hence, if $F$ is compact then, by the
continuity of $d_F$ and Lemma~\ref{LRiem}~(iii), $\Crit(d_F)\cap(d_F)^{-1}[a,b]$ is compact for any $0<a<b<\infty$. Hence, using the continuity of $d_F$ again,  
$$N\cap [a,b]=d_F(\Crit(d_F)\cap(d_F)^{-1}[a,b])$$ 
is compact for any $0<a<b<\infty$. Thus, $N$ is closed in $(0,\infty)$. 

We shall verify now (i) and (ii) for $r\in d_F(M\setminus F)\setminus N$.
By Theorem~\ref{abstd2}, for $r\in d_F(M\setminus F)\setminus N$ and for each $i$,
$$\varphi_i(S_r(F)\cap U_i)=(d_F\circ\varphi_i^{-1})^{-1}(r)$$
is either empty or an $(n-1)$-dimensional semiconcave manifold. As $\varphi_i$ is locally bilipschitz, (i) follows.

To show (ii), note that if $r\in d_F(M\setminus F)\setminus N$ then all points of $S_r(F)$ are regular points of $-d_F$ (in the sense of Bangert). Consider any connected component
$M'$ of $M\setminus F$ and let $f$ be the restriction of $-d_F$ to $M'$. As $f$ is
locally semiconvex, Bangert's result (Ban) cited above implies that
$\{ p\in M':\, f(p)\leq r\}=\{ p\in M':\, d_F(p)\geq r\}$ has locally positive reach
in $M'$. It follows easily that $\{ p\in M:\, d_F(p)\geq r\}$ has locally positive
reach in $M$ as well.

Let $F\subset M$ be compact. Denoting $F_{\geq r}:=\{p\in M:\,
d_F(p)\geq r\}$ for brevity, we have 
$$\reach F_{\geq r}=\min\Big\{\inf_{r\leq d_F(p)\leq 2r}\reach(F_{\geq r},p),
\inf_{d_F(p)>2r}\reach(F_{\geq r},p)\Big\}.$$
The first infimum is positive since $\reach(F_{\geq r},\cdot)$ is continuous and
positive, and $\{p:\, r\leq d_F(p)\leq 2r\}$ is compact. The second infimum is clearly
greater or equal to $r$. Thus, $\reach F_{\geq r}>0$.
\hfill{$\square$}

\begin{remark}
The sets $S_r(F)$, for $r\in d_F(M\setminus F)\setminus N$, are rather regular Lipschitz manifolds. Indeed, our proof gives that they are ``semiconcave surfaces'' in the sense that, for each chart $(U,\vf)$ on $M$, the image of $S_r(F) \cap U$ under $\vf$ is either empty, or a semiconcave surface in $\R^n$. 
\end{remark}

Finally, we apply Theorem~\ref{kazden} to Riemannian manifolds (of arbitrary
dimension).

\begin{theorem}
Let $F$ be a nonempty closed subset of a connected complete smooth  $n$-dimensional Riemannian
manifold $M$ with $n\geq 2$. Then, for all $r\in d_F(M\setminus F)$, up to a countable set, the set $A_r:=
S_r(F)\setminus\Crit(d_F)$ is an $(n-1)$-dimensional Lipschitz manifold which is
open and dense in $S_r(F)$ and ${\cal H}^{n-1}(S_r(F)\setminus A_r)=0$.
\end{theorem}

\begin{proof}
Let $(U,\vf)$ be any chart in $M\setminus F$. Applying Theorem~\ref{kazden} to the
locally semiconcave function $d_F\circ\vf^{-1}$, we obtain a countable set
$N^\vf\subset d_F(U)$ such that whenever $r\in d_F(U)\setminus N^\vf$, then
$$B_r^\vf:=(d_F\circ\vf^{-1})^{-1}(r)\setminus\Crit(d_F\circ\vf^{-1})$$ 
is an $(n-1)$-dimensional semiconcave surface which is open dense in
$(d_F\circ\vf^{-1})^{-1}(r)=\vf(S_r(F)\cap U)$ and fulfills ${\cal
H}^{n-1}(\vf(S_r(F)\cap U)\setminus B_r^\vf)=0$. Since $\vf$ is locally bilipschitz, 
$A_r^\vf:=\vf^{-1}(B_r^\vf)=A_r\cap U$ 
is an $(n-1)$-dimensional Lipschitz manifold, it is open dense in $S_r(F)\cap U$ and ${\cal
H}^{n-1}(S_r(F)\cap U\setminus A_r^\vf)=0$. Considering a countable atlas of
$M\setminus F$, the proof is finished in a standard way.
\end{proof}

\section{Alexandrov spaces}
Let $M$ be an $n$-dimensional Alexandrov space ($n\geq 2$) with lower curvature bound (i.e., $M$ is a complete, locally compact length space with lower curvature bound in the sense of Alexandrov, and with finite Hausdorff dimension $n$, see \cite[Chapter~10]{BBI}).

A point $p\in M$ is called {\it regular} if the space of directions at $p$, $\Sigma_p(M)$, is isometric to the unit sphere $S^{n-1}$. Otherwise, $p\in M$ is called {\it singular}; we denote by $S_M$ the set of all singular points of $M$. 
The set of singular points has Hausdorff dimension at most $n-1$ and if
$X$ has no boundary, then $\dim_HS_M\leq n-2$ (see \cite[\S10.6]{BGP}). If $n=2$ and $M$ has no boundary then $S_M$ is even countable (see \cite[Lemma~1.3]{Machi}).

Perelman \cite{Per} introduced the set $M^*\subset M$ of all points $p\in M$ such that there exist $\xi_1,\ldots ,\xi_{n+1}\in\Sigma_p(M)$ with $\angle(\xi_i,\xi_j)>\pi/2$ for any $1\le i<j\leq n+1$. We shall call the points of $M^*$ ``Perelman regular'', and the remaining points in $M$ ``Perelman singular''. It is well-known (and easy to see) that any regular point is Perelman regular as well.  
Thus, $M\setminus M^*$ is countable if $n=2$ and $M$ has no boundary. Further, $M^*$ is a dense, open and convex subset of $M$ (see \cite[the end of \S3]{Per}).
Perelman introduced and applied a ``DC structure'' on $M^*$. We will need only the following fact about it (see \cite[p.~3, l.~14-15, and Proposition~(C)]{Per}).
\vskip 2mm

(Per)
{\it For any $p\in M^*$ there exists an open neighbourhood $U$ of $p$ in $M$ and a bilipchitz mapping $\varphi: U\to\R^n$ such that $\vf(U)$ is open and, if $f$ is a function on $U$ that is semiconcave in the intrinsic sense, then $f\circ \varphi^{-1}$ is locally DC on $\varphi(U)$.}

\vskip 2mm
Following \cite[\S2.7]{Kuwae}, we shall call the pair $(U,\vf)$ from (Per) a {\it DC local chart}.

Note that semiconcavity in the intrinsic sense is defined by means of semiconcavity along geodesic paths, see \cite[Definition~124]{Pl} for a precise definition.
 
The proof of (Per) is contained only in the unpublished preprint \cite{Per}, but its validity is adopted and used by experts in the theory of Alexandrov spaces
(cf., e.g., \cite{Kuwae}).

As in the previous chapters, we shall use the notation $d_F$ for the distance function from a closed set $F\subset M$.
  
\begin{theorem} \label{T_Alex-2}
Let $n\in\{ 2,3\}$ and let $M$ be an $n$-dimensional Alexandrov space with lower curvature bound and $F$ a closed subset of $M$. Then, the following hold.
\begin{enumerate}
\item[{\rm (i)}]
There exists a set $N\subset (0,\infty)$ with $\Ha^{(n-1)/2}(N)=0$ such that for all $r\in d_F(M^*\setminus F)\setminus N$, $S_r(F)\cap M^*$ is an $(n-1)$-dimensional Lipschitz manifold.
\item[{\rm (ii)}] If, moreover, $\Ha^{(n-1)/2}(M\setminus M^*)=0$ then there exists a set $N'\subset (0,\infty)$ with $\Ha^{(n-1)/2}(N')=0$ such that for all $r\in d_F(M\setminus F)\setminus N'$, $S_r(F)$ is an $(n-1)$-dimensional Lipschitz manifold. If, in addition, $F$ is compact then $N'$ can be chosen to be relatively closed in $(0,\infty)$.
\end{enumerate}
\end{theorem}

\begin{corollary} \label{Cor-2D}
Let $M$ be a two-dimensional Alexandrov space with lower curvature bound and without boundary, and let $F$ be a compact subset of $M$. Then there exists a relatively closed subset $N$ of $(0,\infty)$ with $\Ha^{1/2}(N)=0$ such that for all $r\in d_F(M\setminus F)\setminus N$, $S_r(F)$ is a one-dimensional Lipschitz manifold.
\end{corollary}

\begin{remark}
Corollary~\ref{Cor-2D} improves partially Shiohama's and Tanaka's result \cite[Theorem~B]{ST}, where the exceptional set is of one-dimensional measure zero and need not be closed.
\end{remark}

\begin{proof}
(i) Let $(U,\vf)$ be a DC local chart in $M^*$. Since the distance function $d_F$
is semiconcave on $M\setminus F$ in the intrinsic sense (see \cite[Proposition~125]{Pl}), the composed mapping $d_F\circ\varphi^{-1}$ is locally DC on $\varphi(U)\subset\R^n$ by (Per). 

We shall show that $d_F\circ\vf^{-1}$ has no stationary point. Take a point $p\in U$ and notice that, since $M$ is complete and boundedly compact, there exists at least one $F$-segment emanating from $p$, i.e., a unit-speed geodesic path $\gamma:[0,a]\to M$ such that $\gamma(0)=p$, 
$\gamma(a)\in F$ and $a-t=d_F(\gamma(t))$, $t\in[0,a]$. Then, denoting $x_t:=\vf(\gamma(t))$, we have
$$-t=d_F(\gamma(t))-a=d_F\circ\vf^{-1}(x_t)-d_F\circ\vf^{-1}(x_0).$$
Since $|x_t-x_0|\leq ct,$ where $c>0$ is the Lipschitz constant of $\vf$, we get
$$|d_F\circ\vf^{-1}(x_t)-d_F\circ\vf^{-1}(x_0)|\geq c^{-1} |x_t-x_0|,$$
which shows that $x_0$ cannot be a stationary point of $d_F\circ\vf^{-1}$, as $x_t\to x_0$ ($t\to 0+$).

Thus, we can apply Theorem~\ref{abstd2} and find a set $N\subset (0,\infty)$ with $\Ha^{(n-1)/2}(N)\\=0$ and such that for all $r\in d_F(U)\setminus N$, $\varphi(d_F^{-1}(r))$ is an $(n-1)$-dimensional DC surface. Since $\vf$ is bilipschitz, $S_r(F)\cap U=\vf^{-1}(\varphi(d_F^{-1}(r)))$ is an $(n-1)$-dimensional Lipschitz manifold. Using the separability of $M^*$, we can find a countable family $(U_i,\vf_i)$ of DC local charts such that $\bigcup_iU_i=M^*\setminus F$, apply the above procedure to each of these charts and find a common exceptional set $N\subset (0,\infty)$ with $\Ha^{(n-1)/2}(N)=0$ and such that for all $r\in d_F(M^*\setminus F)\setminus N$, $S_r(F)\cap M^*$ is an $(n-1)$-dimensional  Lipschitz manifold. 

(ii) If $\Ha^{(n-1)/2}(M\setminus M^*)=0$ then, since $d_F$ is Lipschitz, we have 
$$\Ha^{(n-1)/2}(d_F(M\setminus M^*))=0$$ 
as well. Hence, enlarging the exceptional set $N$ to $N':=N\cup d_F(M\setminus M^*)$, we obtain that $\Ha^{(n-1)/2}(N')=0$ and the level set $S_r(F)$ itself is an $(n-1)$-dimensional  Lipschitz manifold for $r\in d_F(M\setminus F)\setminus N'$.

Let now $F$ be compact, in addition, and let $(U_i,\vf_i)$ be the countable atlas of DC local charts covering $M^*\setminus F$, as above. It follows from Theorem~\ref{abstd2} that we can take for the exceptional set $N':=d_F(Q)$, where
$$Q=(M\setminus M^*)\cup\{ p\in M^*\setminus F:\, \forall i,\, p\in U_i\implies\vf_i(p)\in\Crit(d_F\circ\vf_i^{-1})\}.$$
Using that $M^*$ is open, each $\Crit(d_F\circ\vf_i^{-1})$ is closed in $\vf_i(U_i)$ and $\vf_i$ are homeomorphisms, we get that the set $Q$ is closed in $M\setminus F$. If $0<a<b<\infty$, the set $d_F^{-1}[a,b]$ is compact (since $d_F$ is continuous and $M$ is boundedly compact, see \cite[\S10.8]{BBI}) and, hence, 
$$N'\cap [a,b]=d_F(Q\cap d_F^{-1}[a,b])$$
is compact as well. Consequently, $N'$ is closed in $(0,\infty)$.
\end{proof}

\begin{remark}
It is easy to see that in a general (possibly with boundary) three-dimensional (and even two-dimensional) Alexandrov space $M$ with lower curvature bound, Ferry's result (almost all distance spheres are topological manifolds) does not hold (for $M$ we can take a closed ball in $\R^2$ or $\R^3$). However, we do not know whether Ferry's result holds in each three-dimensional Alexandrov space with lower curvature bound and without boundary. In this case our method cannot be used since there exists a three-dimensional convex surface $X$ in $\R^4$ for which $\Ha^1(X \setminus X^*) >0$ (see Example~\ref{Ex-cone}).
       
However, Ferry's result holds in every three-dimensional complete convex surface $X$ in $\R^4$; it is proved in \cite{RZ} without using Perelman's DC structure.
\end{remark}

\begin{example} \label{Ex-cone}
We shall demonstrate on a particular example that all points on one-di\-men\-sional ``sufficiently sharp'' edges 
of three-dimensional convex surfaces are Perelman singular. Let $A=\{ (x,y,z)\in\R^3:\, \alpha^2(x^2+y^2)=z^2, z\geq 0\}$ with $\alpha\geq\sqrt{2\pi^2-1}$ and consider the convex cone $C=A\times\R$ in $\R^4$. Then, for the convex surface $X=\partial C$, any point of the edge $\{(0,0,0)\}\times\R$ is Perelman singular. Of course, it suffices to show that the origin $0$ is Perelman singular. There even do not exist three directions of $X$ at $0$ forming obtuse angles each with other. To see this, let $\xi_1,\xi_2,\xi_3$ be three non-zero vectors from $X$ determining three directions of $X$ at $0$.
Multiplying by positive factors, we can suppose that these vectors can be written as 
$$\xi_i=(r\cos\vartheta_i,r\sin\vartheta_i,\alpha r,s_i),\quad i=1,2,3,$$
where $r\geq 0$, $\vartheta_i\in [0,2\pi)$ and $s_i\in\R$, $i=1,2,3$. At least two of the three numbers $s_i$ must have nonnegative product, so assume without loss of generality that $s_1,s_2\geq 0$. We shall show that the angle formed by the directions of $\xi_1$ and $\xi_2$ on $X$ is not obtuse. 
Since $C$ is a cone, the angle can be obtained as
$$\angle(\xi_1,\xi_2)=\arccos\frac{\|\xi_1\|^2+\|\xi_2\|^2-\dist^2(\xi_1,\xi_2)}{2\|\xi_1\|\|\xi_2\|},$$
cf.\ \cite[\S3.6.5]{BBI} ($\dist$ is the intrinsic distance in $X$).
The points $\xi_1$ and $\xi_2$ can be connected by the following path on $X$
$$\gamma:\, t\mapsto\Big( r\cos(\vartheta_1+t(\vartheta_2-\vartheta_1)), r\sin(\vartheta_1+t(\vartheta_2-\vartheta_1)),\alpha r, s_1+t(s_2-s_1)\Big),\ t\in[0,1]$$
of length
$$\length\gamma=\int_0^1\|\gamma'(t)\|\, dt=\sqrt{(\vartheta_2-\vartheta_1)^2r^2+(s_2-s_1)^2}.$$
Hence,
$$\dist^2(\xi_1,\xi_2)\leq(\length\gamma)^2\leq 4\pi^2r^2+s_1^2+s_2^2$$
which is less or equal to 
$$\|\xi_1\|^2+\|\xi_2\|^2=2(1+\alpha^2)r^2+s_1^2+s_2^2$$
since $\alpha^2\geq 2\pi^2-1$. Hence, the angle formed by $\xi_1$ and $\xi_2$ is not obtuse.
\end{example}

We finish this section by an application of Theorem~\ref{kazden} to Alexandrov spaces of any dimension.

\begin{theorem} \label{kazden-Alex}
Let $M$ be an $n$-dimensional ($n\geq 2$) Alexandrov space with lower curvature bound and without boundary, and let $F\subset M$ be closed. Then, for all $r\in d_F(M\setminus F)$ except a countable set, either $\Ha^{n-1}(S_r(F))=0$, or there exists an $(n-1)$-dimensional Lipschitz manifold $A_r\subset S_r(F)$ which is open in $S_r(F)$ and $\Ha^{n-1}(S_r(F)\setminus A_r)=0$ holds.
\end{theorem} 

\begin{remark}
If, in addition, $M=M^*$ then the manifolds $A_r$ in Theorem~\ref{kazden-Alex} can be found so that there are moreover dense in $S_r(F)$.
\end{remark}    
 
\begin{proof}
Let $(U_i,\vf_i)$ be a countable atlas of DC local charts covering $M^*\setminus F$, as in the proof of Theorem~\ref{T_Alex-2}. Applying Theorem~\ref{kazden} to $d_F\circ\vf_i^{-1}$ (the validity of assumptions was shown in the proof of Theorem~\ref{T_Alex-2}) and the bilipschitz property of $\vf_i$, we obtain countable sets $N_i\subset d_F(U_i)$ such that for all $i$ and all $r\in d_F(U_i)\setminus N_i$,
$$P_i:=(S_r(F)\cap U_i)\setminus\vf_i^{-1}(\Crit (d_F\circ\vf_i^{-1}))$$
is an $(n-1)$-dimensional Lipschitz manifold with $\Ha^{n-1}((S_r(F)\cap U_i)\setminus P_i)=0$. Set $N:=\bigcup_iN_i$ and
for $r\in d_F(M^*\setminus F)\setminus N$,
\begin{eqnarray*}
A_r:&=&\{ p\in S_r(F)\cap M^*:\, \exists \delta>0, S_r(F)\cap B(p,\delta)\text{ is an }(n-1)\text{-dimensional}\\ 
&&\hspace{3cm}\text{Lipschitz manifold}\}.
\end{eqnarray*}
Clearly, $A_r$ is an $(n-1)$-dimensional Lipschitz manifold open in $S_r(F)$. Note that $\bigcup_iP_i\subset A_r$ and recall that $\dim_H(M\setminus M^*)\leq\dim_H(S_M)\leq n-2$. Hence,
$$\Ha^{n-1}(S_r(F)\setminus A_r)\leq\Ha^{n-1}(M\setminus M^*) +\Ha^{n-1}((S_r(F)\cap M^*)\setminus\bigcup\nolimits_iP_i)=0$$
for all $r\in d_F(M^*\setminus F)\setminus N$, as required. If $r\in d_F(M\setminus F)\setminus d_F(M^*\setminus F)$ then $\Ha^{n-1}(S_r(F))\leq\Ha^{n-1}(M\setminus M^*)=0$.
\end{proof}

\end{document}